\documentclass[11pt]{amsart}
\usepackage{MnSymbol}
\usepackage{mathrsfs}
\usepackage[utf8]{inputenc}
\usepackage{amsmath,amsthm}
\usepackage{tikz}
\usetikzlibrary{matrix}

\newtheorem{Theorem}{Theorem}[section]
\newtheorem{Corollary}[Theorem]{Corollary}
\newtheorem{Lemma}[Theorem]{Lemma}
\newtheorem{Proposition}[Theorem]{Proposition}
\theoremstyle{definition}

\numberwithin{equation}{section}

\DeclareMathAlphabet\mathbb{U}{msb}{m}{n}
\newcommand{\mono}{\rightarrowtail}
\newcommand{\epi}{\twoheadrightarrow}

\def\QQ{{\mathbb Q}}
\def\ZZ{{\mathbb Z}}

\def\DL{{\mathcal D\mathcal L}}

\usepackage[all]{xy}
\def\xyma{\xymatrix@M.7em}

\begin{document}

\title{On discrete homology of a free pro-$p$-group}
\author{Sergei O. Ivanov}
\address{Chebyshev Laboratory, St. Petersburg State University, 14th Line, 29b,
Saint Petersburg, 199178 Russia} \email{ivanov.s.o.1986@gmail.com}

\author{Roman Mikhailov}
\address{Chebyshev Laboratory, St. Petersburg State University, 14th Line, 29b,
Saint Petersburg, 199178 Russia and St. Petersburg Department of
Steklov Mathematical Institute} \email{rmikhailov@mail.ru}

\begin{abstract} For a prime $p$, let $\hat F_p$ be a finitely generated free pro-$p$-group of rank $\geq 2$. We show that the second discrete homology group $H_2(\hat F_p,\mathbb Z/p)$ is an uncountable $\mathbb Z/p$-vector space. This answers a problem of A.K. Bousfield.
\end{abstract}
\maketitle

\section{Introduction}

Let $p$ be a prime.
For a profinite group $G$, there is a natural comparison map
$$H_2^{\sf disc}(G,\ZZ/p)\to H_2^{\sf cont}(G,\ZZ/p),$$
which connects discrete and continuous homology groups of $G$. Here $H_2^{\sf disc}(G,\ZZ/p)=H_2(G, \ZZ/p)$ is the second homology group of $G$ with $\ZZ/p$-coefficients, where $G$ is viewed as a discrete group. The continuous homology $H_2^{\sf cont}(G,\ZZ/p)$ can be defined as the inverse limit $\varprojlim H_2(G/U, \ZZ/p),$ where $U$ runs over all open normal subgroups of $G$. The above comparison map $H_2^{\sf disc}\to H_2^{\sf cont}$ is the inverse limit of the coinflation maps $H_2(G,\ZZ/p)\to H_2(G/U, \ZZ/p)$ (see Theorem 2.1 \cite{FKRS}).

The study of the comparison map for different types of pro-$p$-groups is a fundamental problem in the theory of profinite groups (see \cite{FKRS} for discussion and references). 
It is well known that for a finitely generated free pro-$p$-group $\hat F_p$, 
$$H_2^{{\sf cont}}(\hat F_p,\mathbb Z/p)=0.$$ 
A.K. Bousfield posed the following question in~\cite{Bousfield77}, 
Problem 4.11 (case $R=\mathbb Z/n$):\\ \\
{\bf Problem.}\ {\em (Bousfield)}  Does $H_2^{\sf disc}(\hat F_n, \mathbb Z/n)$ vanish when $F$ is a finitely generated free group?\\

\noindent Here $\hat F_n$ is the $\ZZ/n$-completion of $F$, which is isomorphic to the product of pro-$p$-completions $\hat F_p$ over prime factors of $n$ (see Prop. 12.3 \cite{Bousfield77}). That is, the above problem is completely reduced to the case of homology groups $H_2^{\sf disc}(\hat F_p, \mathbb Z/p)$ for primes $p$ and since $H_2^{\sf cont}(\hat F_p,\mathbb Z/p)=0$, the problem becomes a question about non-triviality of the kernel of the comparison map for $\hat F_p$.

In \cite{Bousfield92}, A.K. Bousfield proved that, for a finitely generated free pro-$p$-group $\hat F_p$ on at least two generators, the group $H_i^{\sf disc}(\hat F_p,\mathbb Z/p)$ is uncountable for $i=2$ or $i=3$, or both. In particular, the wedge of two circles $S^1\vee S^1$ is a $\mathbb Z/p$-bad space in the sense of Bousfield-Kan.

The group $H_2^{\sf disc}(\hat F_p, \ZZ/p)$ plays a central role in the theory of $H\ZZ/p$-localizations developed in \cite{Bousfield77}. It follows immediately from the definition of $H\ZZ/p$-localization that, for a free group $F$, $H_2^{\sf disc}(\hat F_p, \ZZ/p)=0$ if and only if $\hat F_p$ coincides with the $H\ZZ/p$-localization of $F$.

In this paper we answer Bousfield's problem over $\ZZ/p$. Our main result is the following

\vspace{.5cm}\noindent{\bf Main Theorem.} For a finitely generated free pro-$p$-group $\hat F_p$ of rank $\geq 2$, $H_2^{\sf disc}(\hat F_p,\mathbb Z/p)$ is uncountable.

\vspace{.25cm}

There are two cases in the problem of Bousfield, $R=\ZZ/n$ and $R=\QQ.$ We give the answer for the case of $R=\ZZ/n$. The problem over $\QQ$ is still open: we do not know whether $H_2^{\sf disc}(\hat F_\QQ,\QQ)$ vanishes or not.

The proof is organized as follows.

In Section 2 we consider properties of discrete and continuous homology of profinite groups. Using a result of Nikolov and Segal {\cite[Th 1.4]{NikolovSegal}}, we show that for a finitely generated profinite group $G$ and a closed normal subgroup $H$ the cokernels of the maps $H_2^{\sf disc}(G,\ZZ/p)\to H_2^{\sf disc}(G/H,\ZZ/p)$ and $H_2^{\sf cont}(G,\ZZ/p)\to H_2^{\sf cont}(G/H,\ZZ/p)$ coincide (Theorem \ref{theorem_cokernels}):
$$ \xymatrix{
H_2^{\sf disc}(G,\ZZ/p)\ar[r] \ar[d]^{\varphi} & H_2^{\sf disc}(G/H,\ZZ/p)\ar[r] \ar[d]^{\varphi} & Q^{\sf disc}\ar[r] \ar[d]^{\cong} & 0
\\
H_2^{\sf cont}(G,\ZZ/p)\ar[r] & H_2^{\sf cont}(G/H,\ZZ/p)\ar[r] & Q^{\sf cont}\ar[r] & 0.
}$$
As a corollary we obtain (Corollary \ref{corollary_main}) that, for a finitely generated free pro-$p$-group $\hat F_p$, a continuous epimorphism $\pi:\hat F_p\epi G$ to a pro-$p$-group induces the exact sequence
\begin{equation}\label{exseq1} H_2^{\sf disc}(\hat F_p,\ZZ/p)\overset{\pi_*}\longrightarrow H_2^{\sf disc}(G,\ZZ/p)\overset{\varphi}\longrightarrow H^{\sf cont}_2(G,\ZZ/p)\longrightarrow 0.\end{equation} That is, for proving that, for a free group $F$, $H_2(\hat F_p, \ZZ/p)\neq 0$, it is enough to find a discrete epimorphism $F\twoheadrightarrow G$ such that, the comparison map of the second homology groups of the pro-$p$-completion of $G$ has a nonzero kernel. Observe that, the statements in Section 2 significantly use the theory of profinite groups and there is no direct way to generalize them for pronilpotent groups. In particular, we don't see how to prove that, $H_2(\hat F_{\ZZ}, \mathbb Z/p)\neq 0$, where $\hat F_{\ZZ}$ is the pronilpotent completion of $F$.

Section 3 follows the ideas of Bousfield from \cite{Bousfield92}. Consider the ring of formal power series $\ZZ/p[\![x]\!]$, and the infinite cyclic group $C:=\langle t\rangle$. We will use the multiplicative notation of the $p$-adic integers $C\otimes \mathbb Z_p=\{t^\alpha,\ \alpha\in \mathbb Z_p\}$. Consider the continuous multiplicative homomorphism $\tau: C\otimes \mathbb Z_p\to \ZZ/p[\![x]\!]$ sending $t$ to $1-x$. The main result of Section 3 is Proposition \ref{proposition_kernel_uncountable}, which claims that the kernel of the multiplication map
\begin{equation}\label{multmap1}\ZZ/p[\![x]\!]\otimes_{\ZZ/p[C\otimes \ZZ_p]} \ZZ/p[\![x]\!] \longrightarrow \ZZ/p[\![x]\!]\end{equation}
 is uncountable.

Our main example is based on the $p$-lamplighter group $\mathbb Z/p\wr \mathbb Z$, a finitely generated but not finitely presented group, which plays a central role in the theory of metabelian groups. The homological properties of the $p$-lamplighter group are considered in \cite{Kropholler}. The profinite completion of the $p$-lamplighter group is considered in \cite{GK}, it is shown there that it is a semi-similar group generated by finite automaton. We consider the
{\it double lamplighter group}:
$$(\ZZ/p)^2 \wr \ZZ =\langle a,b,c \mid [b,b^{a^i}]=[c,c^{a^i}]=[b,c^{a^i}]=b^p=c^p=1 , \ \  i\in \ZZ \rangle,$$
Denote by $\DL$ the pro-$p$-completion of the double lamplighter group. It follows from direct computations of homology groups, that there is a diagram (in the above notations)
$$
\xyma{
\ZZ/p[\![x]\!]\otimes_{\ZZ/p[C\otimes \ZZ_p]} \ZZ/p[\![x]\!]\ar@{->}[r]\ar@{>->}[d]^\oplus & \ZZ/p[\![x]\!]\ar[d]
\\
H_2^{\sf disc}(\DL,\ZZ/p) \ar@{->}[r] & H_2^{\sf cont}(\DL,\ZZ/p),
}
$$
where the left vertical arrow is a split monomorphism and the upper horisontal map is the multiplication map (see proof of Theorem \ref{theorem_DL}). This implies that, for the group $\DL$, the comparison map $ H_2^{\sf disc}(\DL,\ZZ/p) \to H_2^{\sf cont}(\DL,\ZZ/p)$ has an uncountable kernel. Since the double lamplighter group is 3-generated, the sequence (\ref{exseq1}) implies that, for a free group $F$ with at least three generators, $H_2(\hat F_p,\mathbb Z/p)$ is uncountable. Finally, we use Lemma 11.2 of \cite{Bousfield92} to get the same result for a two generated free group $F$.

In \cite{Bousfield77}, A.K. Bousfield formulated the following generalization of the above problem for the class of finitely presented groups (see Problem 4.10 \cite{Bousfield77}, the case $R=\mathbb Z/n$). Let $G$ be a finitely presented group, is it true that $H\mathbb Z/p$-localization of $G$ equals to its pro-$p$-completion $\hat G_p$? [The Problem is formulated for $H\mathbb Z/n$-localization, but it is reduced to the case of a prime $n=p$.] It follows immediately from the definition of $H\mathbb Z/p$-localization that, this problem can be reformulated as follows: is it true that, for a finitely presented group $G$, the natural homomorphism $H_2(G, \mathbb Z/p)\to H_2(\hat G_p, \mathbb Z/p).$ It is shown in \cite{Bousfield77} that this is true for the class of poly-cyclic groups. The same is true for finitely presented metabelian groups \cite{IvanovMikhailov}. Main Theorem of the present paper implies that, for any finitely presented group $P$, which maps epimorphically onto the double lamplighter group, the natural map $H_2(P, \mathbb Z/p)\to H_2(\hat P_p, \mathbb Z/p)$ has an uncountable cokernel.

\section{Discrete and continuous homology of profinite groups}

For a profinite group $G$ and a normal subgroup $H$, denote by $\overline{H}$ the closure of $H$ in $G$ in profinite topology.

\begin{Theorem}[{\cite[Th 1.4]{NikolovSegal}}]\label{NikolovSegal} Let $G$ be a finitely generated profinite group and $H$ be a closed normal subgroup of $G.$ Then the subgroup $[H,G]$ is closed in $G.$
\end{Theorem}

\begin{Corollary}\label{corollary_[H,G]H^p} Let $G$ be a finitely generated profinite group and $H$ be a closed normal subgroup of $G.$ Then the subgroup $[H,G]\cdot H^p$ is closed in $G.$
\end{Corollary}
\begin{proof}
Consider the abelian profinite group $H/[H,G].$  Then the  $p$-power map  $H/[H,G] \to H/[H,G] $ is continuous and its image is equal to $ ([H,G]\cdot H^p)/[H,G].$ Hence
$([H,G]\cdot H^p)/[H,G]$ is a closed subgroup of $ H/[H,G].$ Using that the preimage of a closed set under continuous function is closed we obtain that $[H,G]\cdot H^p$ is closed.
\end{proof}

\begin{Lemma}[mod-$p$ Hopf's formula]\label{lemma_hopf} Let $G$ be a (discrete) group and $H$ be its normal subgroup. Then there is a natural exact sequence
$$H_2(G,\ZZ/p) \longrightarrow H_2(G/H,\ZZ/p) \longrightarrow \frac{H\cap ([G,G] G^p)}{[H,G] H^p} \longrightarrow 0.$$
\end{Lemma}
\begin{proof}
It follows from the five term exact sequence
$$ H_2(G,\ZZ/p) \longrightarrow H_2(G/H,\ZZ/p) \longrightarrow H_1(H,\ZZ/p)_G \longrightarrow H_1(G,\ZZ/p)$$
and the equations $H_1(H,\ZZ/p)_G=H/([H,G]H^p)$ and $H_1(G,\ZZ/p)=G/([G,G]G^p).$
\end{proof}

\begin{Lemma}[profinite mod-$p$ Hopf's formula]\label{lemma_profinite_hopf} Let $G$ be a profinite group and $H$ be its closed normal subgroup. Then there is a natural exact sequence
$$H_2^{\sf cont}(G,\ZZ/p) \longrightarrow H_2^{\sf cont}(G/H,\ZZ/p) \longrightarrow \frac{H\cap \overline{([G,G] G^p)}}{\overline{([H,G] H^p)}} \longrightarrow 0.$$
\end{Lemma}
\begin{proof}
For the sake of simplicity we set $H_*(-)=H_*^{{\sf discr}}(-,\ZZ/p)$  and $H_*^{\sf cont}(-):=H^{\sf cont}_*(-,\ZZ/p).$ Consider the five term exact sequence
 (Corollary 7.2.6 of \cite{RibesZalesskii})
$$  H^{\sf cont}_2(G) \longrightarrow H^{\sf cont}_2(G) \longrightarrow H^{\sf cont}_0(G,H^{\sf cont}_1(H)) \longrightarrow H^{\sf cont}_1(G).$$
Continuous homology and cohomology of profinite groups are Pontryagin dual to each other (Proposition 6.3.6 of \cite{RibesZalesskii}). There are isomorphisms $$H^1_{\sf cont}(G)={\it Hom}(G/\overline{[G,G]G^p},\ZZ/p)={\it Hom}(G/\overline{[G,G]G^p},\QQ/\ZZ),$$ where ${\it Hom}$ denotes the set of continuous homomorphisms (see \cite[I.2.3]{Serre}).   It follows that $H_1^{\sf cont}(G)=G/\overline{[G,G]G^p}.$ Similarly  $H_1^{\sf cont}(H)=H/\overline{[H,H]H^p}.$ Lemma 6.3.3. of \cite{RibesZalesskii} implies that  $H_0^{\sf cont}(G,M)=M/\overline{\langle m-mg\mid m\in M, g\in G  \rangle }$ for any profinite $(\ZZ/p[G])^\wedge$-module $M.$ Therefore $H^{\sf cont}_0(G,H^{\sf cont}_1(H))=H/\overline{[H,H]H^p}.$ The assertion follows.
\end{proof}

We denote by $\varphi$ the comparison map $$\varphi:H_2^{\sf disc}(G,\ZZ/p)\to H_2^{\sf cont}(G,\ZZ/p).$$

\begin{Theorem}\label{theorem_cokernels} Let $G$ be a finitely generated profinite group and $H$ a closed normal subgroup of $G$. Denote
\begin{align*} & Q^{\sf disc}:={\sf Coker}(H_2(G,\ZZ/p)\to H_2(G/H,\ZZ/p))\\ & Q^{\sf cont}:={\sf Coker}(H_2^{\sf cont}(G,\ZZ/p)\to H_2^{\sf cont}(G/H,\ZZ/p)).\end{align*}
 Then the comparison maps $\varphi$ induce an isomorphism
 $Q^{\sf disc}\cong Q^{\sf cont}:$

$$ \xymatrix{
H_2^{\sf disc}(G,\ZZ/p)\ar[r] \ar[d]^{\varphi} & H_2^{\sf disc}(G/H,\ZZ/p)\ar[r] \ar[d]^{\varphi} & Q^{\sf disc}\ar[r] \ar[d]^{\cong} & 0
\\
H_2^{\sf cont}(G,\ZZ/p)\ar[r] & H_2^{\sf cont}(G/H,\ZZ/p)\ar[r] & Q^{\sf cont}\ar[r] & 0
}$$
\end{Theorem}
\begin{proof} This follows from Lemma \ref{lemma_hopf}, Lemma \ref{lemma_profinite_hopf} and Corollary \ref{corollary_[H,G]H^p}.
\end{proof}

\begin{Corollary}\label{corollary_main} Let $G$ be a finitely generated pro-$p$-group and $\pi:\hat F_p\epi G$ be a continuous epimorphism from the pro-$p$-completion of a finitely generated free group $F.$ Then the sequence
$$H_2^{\sf disc}(\hat F_p,\ZZ/p)\overset{\pi_*}\longrightarrow H_2^{\sf disc}(G,\ZZ/p)\overset{\varphi}\longrightarrow H^{\sf cont}_2(G,\ZZ/p)\longrightarrow 0$$
is exact.
\end{Corollary}
\begin{proof}
This follows from Theorem \ref{theorem_cokernels} and the fact that $H_2^{\sf cont}(\hat F_p,\ZZ/p)=0.$
\end{proof}

\section{Technical results about the ring of power series $\ZZ/p[\![x]\!] $}
In this section we follow to ideas of Bousfield written in Lemmas 10.6,  10.7 of \cite{Bousfield92}.
The goal of this section is to prove Proposition \ref{proposition_kernel_uncountable}

We use the following notation: $C=\langle t \rangle$ is the infinite cyclic group; $C\otimes \ZZ_p$ is the group of  $p$-adic integers written multiplicatively as powers of the generator $C\otimes \ZZ_p=\{ t^\alpha\mid \alpha\in \ZZ_p \};$ $\ZZ/p[\![x]\!]$ is the ring of power series;  $\ZZ/p(\!(x)\!)$ is the field of formal Laurent series.

\begin{Lemma}\label{lemma_dense} Let ${\sf A}$ be a subset of $\ZZ/p[\![x]\!].$ Denote by ${\sf A}^i$ the image of ${\sf A}$ in $\ZZ/p[x]/(x^{p^i}).$ Assume that
$$ \lim_{i\to \infty} |{\sf A}^i|/p^{p^{i}}=0. $$ Then the interior of $\ZZ/p[\![x]\!]\setminus {\sf A}$ is dense in $\ZZ/p[\![x]\!].$
\end{Lemma}
\begin{proof}
Take any power series $f$ and any its neighbourhood of the form $f+(x^{p^s}).$ Then for any $i$ the open set $f+(x^{p^s})$ is the disjoint union of smaller open sets $\bigcup_{t=1}^{p^{p^i}} f+f_t+(x^{p^{s+i}}),$ where $f_t $ runs over representatives of $(x^{p^s})/(x^{p^{s+i}}).$ Chose $i$ so that $|{\sf A}^{s+i}|/ p^{p^{i+s}} \leq p^{-p^s}.$ Then $|{\sf A}^{s+i}| \leq p^{p^i}.$ Hence the number of elements in ${\sf A}^{i+s}$ is lesser the number of open sets $f+f_t+(x^{p^{s+i}})$. It follows that there exists $t$ such that ${\sf A}\cap (f+f_t+(x^{p^{s+i}})) =\emptyset.$ The assertion follows.
\end{proof}

Denote by  $$\tau: C\otimes \ZZ_p\to \ZZ/p[\![x]\!]$$  the continuous multiplicative homomorphism sending $t$ to $1-x.$ It is well defined because $(1-x)^{p^i}=1-x^{p^i}.$

\begin{Lemma}\label{lemma_Laurent} Let   $K$ be the subfield of $\ZZ/p(\!(x)\!)$ generated by the image of $\tau$. Then the degree of the extension $[\ZZ/p(\!(x)\!):K]$ is uncountable.
\end{Lemma}
\begin{proof} Denote the image of the map $\tau:C\otimes \ZZ_p\to \ZZ/p[\![x]\!]$ by ${\sf A}.$ Let $\alpha=(\alpha_1,\dots,\alpha_n)$ and $\beta=(\beta_1,\dots,\beta_n)$, where $\alpha_1,\dots,\alpha_n,\beta_1,\dots,\beta_n\in \ZZ/p,$ and $k\geq 1.$ Denote by ${\sf A}_{\alpha,\beta,k}$ the subset of $\ZZ/p[\![x]\!]$ consisting of elements that can be written in the form
\begin{equation}\label{eq_ratio}
\frac{\alpha_1a_1+\dots +\alpha_na_n}{\beta_1b_1+\dots+\beta_nb_n},
\end{equation}
where $a_1,\dots,a_n,b_1,\dots,b_n\in {\sf A}$ and $\beta_1b_1+\dots+\beta_nb_n \notin (x^{p^k}).$
Then $K\cap \ZZ/p[\![x]\!]= \bigcup_{\alpha,\beta,k}{\sf A}_{\alpha,\beta,k}.$

Fix some $\alpha,\beta,k.$ Take $i\geq k$ and consider the images of ${\sf A}$ and  ${\sf A}_{\alpha,\beta,k}$ in $\ZZ/p[x]/(x^{p^i}).$ Denote them by $ {\sf A}^i $ and $ {\sf A}_{\alpha,\beta,k}^i.$ Obviously ${\sf A}^i $ is the image of the map $C/C^{p^i} \to \ZZ/p[x]/(x^{p^i})$ that sends $t$ to $1-x.$ Then ${\sf A}^i$  consists of $p^i$ elements. Fix some elements $\bar a_1,\dots,\bar a_n,\bar b_1,\dots,\bar b_n\in {\sf A}^i$  that have  preimages $a_1,\dots,a_n,b_1,\dots,b_n \in {\sf A}$ such that the ratio  \eqref{eq_ratio} is in ${\sf A}_{\alpha,\beta,k}.$ For any such preimages $a_1,\dots,a_n,b_1,\dots,b_n \in {\sf A}$ the image $\bar r$ of the ratio \eqref{eq_ratio} satisfies the equation
$$\bar r\cdot (\beta_1\bar b_1+\dots+\beta_n\bar b_n)=\alpha_1\bar a_1+\dots \alpha_n\bar a_n.$$ Since $\beta_1\bar b_1+\dots+\beta_n\bar b_n \notin (x^{p^k}),$ the annihilator of $\beta_1\bar b_1+\dots+\beta_n\bar b_n$ consist of no more than $p^{p^k}$ elements  and the equation has no more than $p^{p^{k}}$ solutions.  Then we have no more than $p^{2in}$ variants of collections $\bar a_1,\dots,\bar a_n,\bar b_1,\dots,\bar b_n\in {\sf A}^i$ and for any such variant there are no more than $p^{p^k}$ variants for the image of the ratio.  Therefore
\begin{equation}\label{eq_ineq1}
|{\sf A}^i_{\alpha,\beta,k}|\leq p^{2in+p^k}.
\end{equation}

Take any sequence of elements $v_1, v_2,\dots\in \ZZ/p(\!(x)\!)$ and prove that  $\sum_{m=1}^\infty Kv_m \ne \ZZ/p(\!(x)\!).$ Note that $x\in K$ because $t\mapsto 1-x.$ Multiplying the elements $v_1,v_2,v_3,\dots $ by powers of $x,$ we can assume that $v_1, v_2,\dots\in \ZZ/p[\![x]\!].$ Fix some $\alpha,\beta,k$ as above. Set $${\sf A}_{\alpha,\beta,k,l}={\sf A}_{\alpha,\beta,k}\cdot v_1+\dots +{\sf A}_{\alpha,\beta,k}\cdot v_l.$$ Then  $\sum_{m=1}^\infty Kv_m = \bigcup_{\alpha,\beta,k,l,j} {\sf A}_{\alpha,\beta,k,l}\cdot x^{-j}.$ Denote by ${\sf A}_{\alpha,\beta,k,l}^i$ the image of ${\sf A}_{\alpha,\beta,k,l}$ in $\ZZ/p[x]/(x^i).$ Then \eqref{eq_ineq1} implies
$ |{\sf A}_{\alpha,\beta,k,l}^i| \leq p^{(2in+p^k)l}.$ Therefore $$\lim_{i\to \infty} |{\sf A}_{\alpha,\beta,k,l}^i|/p^{p^{i}} =0.$$
By Lemma \ref{lemma_dense} the interior  of the complement of ${\sf A}_{\alpha,\beta,k,l}$ is dense in $\ZZ/p[\![x]\!]$.  By the Baire theorem $\sum_{m=1}^\infty Kv_m  = \bigcup_{\alpha,\beta,k,l,j} {\sf A}_{\alpha,\beta,k,l}\cdot x^{-j}$ has empty interior. In particular $\sum_{m=1}^\infty Kv_m \ne \ZZ/p(\!(x)\!).$
\end{proof}

\begin{Proposition}\label{proposition_kernel_uncountable} Consider the ring homomorphism $\ZZ/p[C\otimes \ZZ_p]\to \ZZ/p[\![x]\!]$ induced by $\tau$. Then the kernel of the multiplication map
\begin{equation}\label{eq_multiplication_map}
\ZZ/p[\![x]\!]\otimes_{\ZZ/p[C\otimes \ZZ_p]} \ZZ/p[\![x]\!] \longrightarrow \ZZ/p[\![x]\!]
\end{equation}
is uncountable.
\end{Proposition}
\begin{proof}
As in Lemma \ref{lemma_Laurent} we denote by $K$ the subfield of $\ZZ/p(\!(x)\!) $ generated by the image of $C\otimes \ZZ_p.$ Since $t\mapsto 1-x,$ we have $x,x^{-1}\in K.$ Set $R:=K\cap \ZZ/p[\![x]\!].$ Note that the image of $\ZZ/p[C\otimes \ZZ_p]$ lies in $R.$ Consider the multiplication map $$\mu:\ZZ/p[\![x]\!]\otimes_R K \to \ZZ/p(\!(x)\!).$$ We claim that this  is an isomorphism. Construct the map in the inverse direction $$\kappa: \ZZ/p(\!(x)\!) \to \ZZ/p[\![x]\!]\otimes_R K$$ given by $$ \kappa(\sum_{i=-n}^\infty \alpha_ix^i)=\sum_{i=0}^\infty \alpha_{i+n} x^i \otimes x^{-n}.$$ Since we have $ax\otimes b=a\otimes xb,$  $\kappa$ does not depend on the choice of $n$, we just have to chose it big enough. Using this we get that $\kappa$ is well defined. Obviously $\mu\kappa={\sf id}. $ Chose $a\otimes b\in \ZZ/p[\![x]\!]\otimes_R K.$ Then $b=b_1b_2^{-1},$ where $b_1,b_2\in R.$ Since $b_2 $ is a power series, we can chose $n$ such that $b_2=x^nb_3,$ where $b_3$ is a power series with nontrivial constant term. Then $b_3$ is invertible in the ring of power series and $b_3,b_3^{-1}\in R$ because $x\in K.$  Hence $a\otimes b=ab_3^{-1}b_3 \otimes b=ab_3^{-1}\otimes x^{-n}b_1=ab_1b_3^{-1}\otimes x^{-n}.$ Using this presentation, we see that $\kappa\mu={\sf id}.$ Therefore
\begin{equation}\label{eq_K_iso}
\ZZ/p[\![x]\!]\otimes_R K \cong \ZZ/p(\!(x)\!).
\end{equation}

Since the image of $\ZZ/p[C\otimes \ZZ_p]$ lies in $R,$ the tensor product $\ZZ/p[\![x]\!]\otimes_R \ZZ/p[\![x]\!]$ is a quotient of the tensor product $\ZZ/p[\![x]\!]\otimes_{\ZZ/p[C\otimes \ZZ_p]} \ZZ/p[\![x]\!]$ and it is enough to prove that the kernel of
\begin{equation}\label{eq_multiplication_map2}
\ZZ/p[\![x]\!]\otimes_R \ZZ/p[\![x]\!] \longrightarrow \ZZ/p[\![x]\!]
\end{equation}
is uncountable.

 For any ring homomorphism $R\to S$ and any $R$-modules $M,N$ there is an isomorphism $(M\otimes_R N)\otimes_R S=(M\otimes_R S)\otimes_S (N\otimes_R S).$ Using this and the isomorphism \eqref{eq_K_iso}, we obtain that after application of $-\otimes_R K$ to  \eqref{eq_multiplication_map2}  we have
\begin{equation}\label{eq_multiplication_map3}
 \ZZ/p(\!(x)\!)\otimes_K \ZZ/p(\!(x)\!) \longrightarrow \ZZ/p(\!(x)\!).
\end{equation}
 Assume the contrary that the kernel of the map \eqref{eq_multiplication_map2} is countable (countable = countable or finite). It follows that  the linear map \eqref{eq_multiplication_map3} has countably dimensional kernel.
 Finally, note that the homomorphism $$\Lambda^2_K \ZZ/p(\!(x)\!) \to \ZZ/p(\!(x)\!)\otimes_K \ZZ/p(\!(x)\!)$$ given by  $a\wedge b\mapsto a\otimes b- b\otimes a$ is a monomorphism, its image lies in the kernel and the dimension of  $\Lambda^2_K \ZZ/p(\!(x)\!)$ over $K$ is uncountable because $[\ZZ/p(\!(x)\!):K]$ is uncountable (Lemma \ref{lemma_Laurent}). A contradiction follows.
\end{proof}

\section{Double lamplighter pro-$p$-group}

 Let $A$ be a finitely generated free abelian group written multiplicatively; $\ZZ/p[A]$ be its group algebra; $I$ be its augmentation ideal and $M$ be a $\ZZ/p[A]$-module. Then we denote by $\hat M=\varprojlim M/MI^i$ its $I$-adic completion. We embed $A$ into the pro-$p$-group $A\otimes \ZZ_p.$ We use the following 'multiplicative' notation $a^\alpha:=a\otimes \alpha$ for $a\in A$ and $\alpha\in \ZZ_p.$ Note that for any $a\in A $ the power $a^{p^i}$ acts trivially on $ M/MI^{p^i}$ because $1-a^{p^i}=(1-a)^{p^i}\in I^{p^i}.$
Then we can extend the action of $A$ on $\hat M$ to the action of $A\otimes \ZZ_p$ on $\hat M$ in a continuous way.

The proof of the following lemma can be found in \cite{IvanovMikhailov} but we add it here for completeness.

\begin{Lemma}\label{lemma_homology_of_Z_p} Let $A$ be a finitely generated free abelian group and $M$ be a finitely generated $\ZZ/p[A]$-module. Then
$$H_*(A,M)\cong H_*(A,\hat M) \cong H_*(A\otimes \ZZ_p,\hat M).$$
\end{Lemma}
\begin{proof}
The first isomorphism is proven in \cite{BrownDror}. Since $\ZZ/p[A]$ in Noetheran, it follows that $H_n(A,\hat M)$ is a finite $\ZZ/p$-vector space for any $n$. Prove the second isomorphism. The action of $A\otimes \ZZ_p$ on $\hat M$ gives an action of $A\otimes \ZZ_p$ on $H_*(A,\hat M)$ such that $A$ acts trivially on $H_*(A,\hat M).$ Then we have a homomorphism from $A\otimes \ZZ/p$ to a finite group of automorphisms of $H_n(A,\hat M),$ whose kernel contains $A.$ Since any subgroup of finite index in $A\otimes\ZZ_p$ is open (see Theorem 4.2.2 of \cite{RibesZalesskii}) and $A$ is dense in $A\otimes \ZZ_p,$ we obtain that the action of $A\otimes \ZZ_p$ on $H_*(C,\hat M)$ is trivial.
 Note that $\ZZ_p/\ZZ$ is a divisible torsion free abelian group, and hence $A\otimes (\ZZ_p/\ZZ)\cong \QQ^{\oplus {\bf c}},$ where ${\bf c}$ is the continuous cardinal.  Then the second page of the spectral sequence of the short exact sequence $ A\mono A\otimes \ZZ_p \epi \QQ^{\oplus {\bf c}}$ with coefficients in $\hat M$ is $H_n(\QQ^{\oplus {\bf c}},H_m(A,\hat M)),$ where $L_m:=H_m(A,\hat M) $ is a trivial $\ZZ/p[\QQ^{\oplus {\bf c}}]$-module. Then by universal coefficient theorem we have
 $$0\longrightarrow \Lambda^n(\QQ^{\oplus {\bf c}}) \otimes L_m \longrightarrow H_n(\QQ^{\oplus {\bf c}},L_m) \longrightarrow {\sf Tor}(\Lambda^{n-1}(\QQ^{\oplus {\bf c}}), L_m) \longrightarrow 0.$$ Since $\Lambda^n(\QQ^{\oplus {\bf c}})$ is torsion free and $L_m$ is a $\ZZ/p$-vector space, we get
 $\Lambda^n(\QQ^{\oplus {\bf c}}) \otimes L_m=0$ and ${\sf Tor}(\Lambda^{n-1}(\QQ^{\oplus {\bf c}}), L_m)=0.$ It follows that $H_n(\QQ^{\oplus {\bf c}},L_m)=0$ for $n\geq 1$ and $H_0(\QQ^{\oplus {\bf c}},L_m)=L_m.$ Then the spectral sequence consists only on of one column, and hence $H_*(A\otimes \ZZ_p,\hat M)=H_*(A,\hat M).$
\end{proof}

\begin{Lemma}\label{lemma_tensor_coinvariants} Let $A$ be an abelian group, $M$ be a $\ZZ[A]$-module and $\sigma_M:M\to M$ be an automorphism of the underlying abelian group such that $\sigma_M(ma)=\sigma_M(m)a^{-1}$ for any $m\in M$ and $a\in A.$ Then there is an isomorphism $$(M\otimes M)_A\cong M\otimes_{\ZZ[A]} M$$
given by $$m\otimes m' \leftrightarrow m\otimes \sigma_M(m').$$
\end{Lemma}
\begin{proof}
Consider the isomorphism $\Phi:M\otimes M \to M\otimes M$ given by $\Phi(m\otimes m')=m\otimes \sigma(m').$ The group of coinvariants $(M\otimes M)_A$ is the quotient of $M\otimes M$ by the subgroup $R$ generated by elements $ma\otimes m'a - m\otimes m',$ where $a\in A$ and $m,m'\in M.$ We can write the generators of $R$ in the following form:  $ma\otimes m'$ $-$ $m\otimes m'a^{-1}.$ Then $\Phi(R)$ is generated by  $ma\otimes \sigma_M(m') $ $-$ $m\otimes \sigma_M(m')a.$ Using that $\sigma_M$ is an automorphism, we can rewrite the generators of $R$ as follows: $ma\otimes m' - m\otimes m'a. $ Taking linear combinations of the generators of $\Phi(R)$ we obtain that $\Phi(R)$ is generated by elements $m\lambda\otimes m'-m\otimes m' \lambda,$ where $\lambda\in \ZZ[A].$ Then $(M\otimes M)/\Phi(R)=M\otimes_{\ZZ[A]} M.$
\end{proof}

The group $C=\langle t \rangle$ acts on $\ZZ/p[\![x]\!]$ by multiplication on $1-x.$ As above, we can extend the action of $C$ on $\ZZ[\![x]\!]$ to the action of $C\otimes \ZZ_p$ in a continuous way. The group $$\mathbb Z/p\wr C=\ZZ/p[C] \rtimes C = \langle a,b \mid [b,b^{a^i}]=b^p=1,  \ i\in \ZZ  \rangle$$ is called the lamplighter group. We consider the 'double version' of this group, {\it double lamplighter group}:
$$(\ZZ/p[C]\oplus \ZZ/p[C])  \rtimes C=\langle a,b,c \mid [b,b^{a^i}]=[c,c^{a^i}]=[b,c^{a^i}]=b^p=c^p=1 , \ \  i\in \ZZ \rangle.$$
Its pro-$p$-completion is equal to the semidirect product
$$\DL=(\ZZ/p[\![x]\!] \oplus \ZZ/p[\![x]\!]) \rtimes (C\otimes \ZZ_p),$$
with the described above action of $C\otimes \ZZ_p$ on $\ZZ/p[\![x]\!]$ (see Proposition 4.12 of \cite{IvanovMikhailov}). We call the group $\DL$ {\it double lamplighter pro-$p$-group}.

\begin{Theorem}\label{theorem_DL} The kernel of the comparison homomorphism for the double lamplighter pro-$p$-group
$$\varphi: H_2^{\sf disc}(\DL,\ZZ/p) \longrightarrow  H_2^{\sf cont}(\DL,\ZZ/p)$$
is uncountable.
\end{Theorem}
\begin{proof} For the sake of simplicity we set $H_2(-)=H_2(-,\ZZ/p)$ and $H_2^{\sf cont}(-)=H_2^{\sf cont}(-,\ZZ/p).$
Consider the homological spectral sequence $E$ of the short exact sequence $\ZZ/p[\![x]\!]^2 \mono \DL \epi C\otimes \ZZ_p.$ Then the zero line of the second page is trivial $E^2_{k,0}=H_k(C\otimes \ZZ_p)=(\Lambda^k \ZZ_p)\otimes \ZZ/p=0$ for $k\geq 2.$ Using Lemma \ref{lemma_homology_of_Z_p} we obtain  $H_k(C\otimes\ZZ_p, \ZZ/p[\![x]\!])=H_k(C,\ZZ/p[C])=0$ for $k\geq 1,$ and hence $E^2_{k,1}=0$ for $k\geq 1.$ It follows that
\begin{equation}\label{eq_h_2_dl}
H_2(\DL)=E^2_{0,2}.
\end{equation}
For any $\ZZ/p$-vector space $V$ K\"unneth formula gives  a natural isomorphism
$$H_2(V\oplus V) \cong  (V\otimes V) \oplus H_2(V)^2.$$
 Then we have a split monomorphism
\begin{equation}\label{eq_mono_to_h2_1}
(\ZZ/p[\![ x ]\!]\otimes \ZZ/p[\![ x ]\!])_{C\otimes \ZZ_p} \mono E^2_{0,2}=H_2(\DL).
\end{equation}

It easy to see that the groups $\DL_{(i)}=((x^i)\oplus (x^i))\rtimes (C \otimes p^{i} \ZZ_p)$ form a fundamental system of open normal subgroups. Consider the quotients $\DL^{(i)}=\DL/\DL_{(i)}.$ Then $$H_2^{\sf cont}(\DL)=\varprojlim H_2(\DL^{(i)}).$$ The short exact sequence $\ZZ/p[\![x]\!]^2 \mono \DL \epi C\otimes \ZZ_p$ maps onto the short exact sequence $(\ZZ/p[x]/(x^i))^2 \mono \DL^{(i)} \epi C/C^{p^{i}}.$ Consider the morphism of corresponding spectral sequences $E\to {}^{(i)}\!E.$ Using \eqref{eq_h_2_dl} we obtain
$${\rm Ker}(H_2(\DL)\to H_2(\DL^{(i)}))\supseteq {\rm Ker}(E^2_{2,0} \to {}^{(i)}\!E^2_{2,0}).$$
Similarly to \eqref{eq_mono_to_h2_1} we have a split monomorphism
$$(\ZZ/p[ x ]/(x^i)\otimes \ZZ/p[ x ]/(x^i) )_{C\otimes \ZZ_p} \mono {}^{(i)}E^2_{2,0}.$$
Then we need to prove that the kernel of the map
\begin{equation}\label{eq_map_of_coinvariants}
(\ZZ/p[\![x]\!]\otimes \ZZ/p[\![x]\!])_{C\otimes \ZZ_p}\longrightarrow \varprojlim (\ZZ/p[x]/(x^i)\otimes \ZZ/p[x]/(x^i))_{C\otimes \ZZ_p}
\end{equation}
is uncountable.

Consider the antipod $\sigma:\ZZ/p[C]\to \ZZ/p[C]$ i.e. the ring homomorphism given by $\sigma(t^n)=t^{-n}.$ The antipod induces a homomorphism $ \sigma:\ZZ/p[x]/(x^{i})\to \ZZ/p[x]/(x^i) $ such that $\sigma(1-x)=1+x+x^2+\dots.$ It induces the continuous homomorphism $ \sigma:\ZZ/p[\![x]\!] \to \ZZ/p[\![x]\!]$  such that $\sigma(x)=-x-x^2-\dots.$ Moreover, we consider the antipode $\sigma$ on $\ZZ/p[C\otimes \ZZ_p].$ Note that the homomorphisms $$\ZZ/p[C]\to\ZZ/p[C\otimes \ZZ_p] \to \ZZ/p[\![x]\!] \to \ZZ/p[x]/(x^i)$$ commute with the antipodes.

By Lemma \ref{lemma_tensor_coinvariants} the  correspondence  $a\otimes b\leftrightarrow a\otimes \sigma(b)$ gives isomorphisms
$$(\ZZ/p[\![x]\!]\otimes \ZZ/p[\![x]\!])_{C\otimes \ZZ_p}\cong \ZZ/p[\![x]\!] \otimes_{\ZZ/p[C\otimes \ZZ_p]} \ZZ/p[\![x]\!], $$
$$(\ZZ/p[x]/(x^i)\otimes \ZZ/p[x]/(x^i))_{C\otimes \ZZ_p}\cong \ZZ/p[x]/(x^i) \otimes_{\ZZ/p[C\otimes \ZZ_p]} \ZZ/p[x]/(x^i).$$
Moreover, since $\ZZ/p[C\otimes \ZZ_p]\to \ZZ/p[x]/(x^i)$ is an epimorphism, we obtain
$$\ZZ/p[x]/(x^i) \otimes_{\ZZ/p[C\otimes \ZZ_p]} \ZZ/p[x]/(x^i)\cong \ZZ/p[x]/(x^i).$$
Therefore the homomorphism \eqref{eq_map_of_coinvariants} is isomorphic to the multiplication homomorphism
$$\ZZ/p[\![x]\!]\otimes_{\ZZ/p[C\otimes \ZZ_p]} \ZZ/p[\![x]\!] \longrightarrow \ZZ/p[\![x]\!],$$
whose kernel is uncountable by Proposition  \ref{proposition_kernel_uncountable}.
\end{proof}

\section{Proof of Main Theorem}

Since the double lamplighter pro-$p$-group is $3$-generated, we have a continuous epimorphism  $\hat F_p\epi \DL,$ where $F$ is the $3$-generated free group. Then the statement of the theorem for three generated free group follows from Proposition \ref{theorem_DL} and Corollary \ref{corollary_main}. Using that the $3$-generated free group is a retract of $k$-generated free group for $k\geq 3$, we obtain the result for $k\geq 3.$ The result for two generated free group follows from Lemma 11.2 of \cite{Bousfield92}.

\section*{Acknowledgement}
The research is supported by the Russian Science Foundation grant
N 16-11-10073.

\end{document}